\documentclass[journal,twoside,web]{./cls/ieeecolor}

\usepackage{./cls/lcsys}
\usepackage{cite}
\usepackage{amsmath,amssymb,amsfonts}
\usepackage{algorithmic}
\usepackage{graphicx}
\usepackage{textcomp}
\usepackage{mathtools}
\usepackage{lipsum}
\usepackage{tikz}

\usepackage{cases}
\usepackage[hidelinks]{hyperref} 
\hypersetup{
    pdftitle=LMI-based approximate dynamic programming for three-phase PMSMs}
    
\newcommand{\be}{\begin{equation}}
\newcommand{\ee}{\end{equation}}

\newcommand{\bald}{\begin{aligned}}
\newcommand{\eald}{\end{aligned}}

\newcommand{\bbm}{\begin{bmatrix}}
\newcommand{\ebm}{\end{bmatrix}}

\newcommand{\bbms}{\begin{bsmallmatrix}}
\newcommand{\ebms}{\end{bsmallmatrix}}

\newcommand{\R}{\mathbb{R}}

\newtheorem{lemma}{Lemma}

\newtheorem{coro}{Corollary}

\newtheorem{prop}{Proposition}

\newcommand{\omre}{\omega_{\text{re}}}

\newcommand{
\ttre}{\theta_{\mathrm{re}}}
\newcommand{\pl}[1]{#1(k+1)}
\newcommand{\ks}[1]{#1(k)}

\definecolor{correctionColor}{HTML}{FF00FF}

\definecolor{redBordeaux}{HTML}{632121}

\usepackage{tikz}
\usepackage{color}
\usepackage{multirow}
\usepackage{amssymb}
\usepackage{gensymb}
\usepackage{tabularx}
\usepackage{extarrows}
\usetikzlibrary{fadings}
\usetikzlibrary{patterns}
\usetikzlibrary{shadows.blur}
\usetikzlibrary{shapes}

\usepackage{pgfplots}
\pgfplotsset{width=10cm,compat=1.9}

\definecolor{midnightGreen}{RGB}{11,85,99}
\definecolor{celestialBlue}{RGB}{82, 153, 211}
\usetikzlibrary{arrows.meta}

\begin{document}
\title{Computationally Efficient Near-Optimal Control for Current Ripple Reduction 
and Optimization\\of Three-Phase Motors via LMIs }
\author{Huu-Thinh Do, Trung B. Tran, Jing Sun, Ilya Kolmanovsky 
\thanks{
H.T Do, T.B. Tran and I. Kolmanovsky
are with the
Department of
Aerospace Engineering; J. Sun is with
the Department of Naval Architecture and Marine Engineering, University of Michigan, Ann Arbor, 48109, MI, USA (e-mail: huuthind@umich.edu) }
\thanks{The authors thank Professor Heath Hofmann and Jordan Perks,
for impeccable guidance on motor operation and modeling
and the University of Michigan's EV Center for funding this work. H-T. Do is grateful to Prof. Uday V. Shanbhag for his instruction in approximate dynamic programming.}
}

\maketitle
\thispagestyle{empty}

\begin{abstract}
The optimal control of  three-phase permanent-magnet synchronous motors (PMSMs) is challenging due to their nonlinearity and the discrete nature of the control set.
Existing approaches either rely on mixed-integer trajectory optimization or require computationally intensive value-iteration procedures. This paper proposes a Linear Matrix Inequality (LMI)-based method for approximating the infinite-horizon value function using a quadratic parameterization and iterated Bellman inequalities, yielding a tractable convex program. The computed function can be obtained efficiently offline and used online as a tail cost in a horizon-one optimal control law. Simulation results show that the proposed approach achieves a favorable trade-off between switching effort and current ripple, with performance comparable to that of finite-control-set MPC but with a significantly lower computational cost. 
\end{abstract}


\vspace{-1.2em}

\section{Introduction}
\label{sec:Intro}

The optimization of system configuration and its design parameters for many engineering systems has to increasingly account for the achievable performance by the control system.
Characterizing such performance ceiling falls naturally into the domain of optimal control \cite{SussmannWillems}. Computationally efficient methods are therefore desired to determine optimal or near-optimal operating policies that can facilitate rapid assessment of multiple hardware and design parameter variations.

In this paper, we focus on the case of three-phase permanent synchronous motors (PMSMs) driven by voltage-source inverters (VSIs), and consider how control-dependent performance limits in this application can be characterized efficiently by approximating the value function in the underlying dynamic programming formulation \cite{bellman1966dynamic,bertsekas2012dynamic} for the optimal control problem. The resulting method supports rapid assessment of system configuration and design parameters, important for motor deployment in modern electric vehicles.  Notably, and as we also demonstrate, such an approximate value function can be integrated as the terminal cost into a one-step Model Predictive Control (MPC) to achieve online performance in terms of current ripple and switching frequency comparable to that of finite-control-set (FCS) MPC \cite{zafra2023long,karamanakos2019guidelines,li2022finite} but with a lower online computational burden.

The control of PMSMs has been extensively studied from various perspectives \cite{karamanakos2019guidelines}, including geometric control 
\cite{thounthong2018nonlinear} and energy-based method 
\cite{uddin2019port}. A common premise in these approaches is that the input voltage can be continuously varied. In practice, this assumption is justified because the commanded voltage can be approximated through modulation by switching between discrete voltage levels at a sufficiently high rate. Such a cascade structure is typically referred to as \textit{indirect control}. However, in battery-powered electric vehicles, these modulation schemes can be inefficient from an energy-performance standpoint, as the modulation requires the power electronics to switch at a constant rate regardless of necessity. Furthermore, optimal modulation for multilevel inverters remains {insufficiently addressed} due to the combinatorial nature and the redundancy of the possible switching sequences.
These limitations motivate the development of switching strategies that directly determine the inverter switching actions, commonly referred to as \textit{direct control}.

In the literature, many recent studies have addressed optimality characterizations from both the trajectory-optimization and the value-function-approximation viewpoints.
On the one hand, to account for the rotating behavior, continuous-time switching rules have been proposed in \cite{deaecto2023trajectory,egidio2022trajectory}, but {they require} an arbitrarily fast switching {capability}.
As a result, they are not well-suited for assessing the trade-off between switching losses and tracking performance.
In 
\cite{xu2025finite}, {an approach based on mixed-integer programming has been proposed for regulating}
systems with a finite control set toward a periodic trajectory instead of an equilibrium point. While the setting is generic, it requires readily available periodic state-input trajectories, which are difficult to obtain in practice and, to the best of our knowledge, they have not been obtained for the FCS problem of PMSMs.

On the other hand, {different approaches have been proposed based on approximating the value function in the dynamic programming setting.}
For instance, one of the most ``brute-force" techniques is to discretize the state space into a finite grid and consider the problem as finding the value function for the discrete state problem, e.g., via value iteration \cite{zhu2024global}. The framework is conceptually straightforward to implement, yet the computational burden (time and storage) grows significantly with the system dimension and the grid's granularity. Another approach is to find the best approximation within a space {defined} by a certain set of basis functions, such as polynomials \cite{heydari2016optimal}. In this way, the optimality constraint, or the Bellman equality, can be imposed directly on the function coefficients. However, as the constraint needs to be valid continuously for all the states, one faces the challenge of infinitely many constraints in such a formulation. To relax this problem, the Bellman equality constraint is replaced with an inequality, resulting in an under-approximation of the function, and the conservativeness can be reduced by iteratively enforcing the inequality \cite{wang2015approximate,stellato2016high}.
For linear dynamics, sufficient conditions for such inequalities can be formulated as linear matrix inequalities (LMIs) 
when seeking an approximation in the class of quadratic basis functions; these LMIs can be efficiently solved via convex optimization.


Motivated by the setup in{\cite{wang2015approximate,stellato2016high}}, in this work, we:
\begin{itemize}
    \item propose a quadratic parameterization of the approximate value function for minimizing the motor's current ripple;
    \item derive LMI-based conditions that enforce the Bellman inequality for all values of the motor rotational angle;
    \item develop a verification procedure for the practical stability of the closed-loop system.
\end{itemize}
Based on this formulation, we define a computationally efficient switching law by minimizing the one-step predicted value of the value function approximation.  Simulation results show that offline function generation is significantly faster than the conventional value-iteration technique. Moreover, the closed-loop performance is better than that of the indirect method and is comparable to that of FCS-MPC, with, again, a considerably lower computational burden.
The remainder of the paper is structured as follows. Section \ref{sec:prelim} presents the PMSM model and the optimal control setting.
The reformulation of the value function approximation problem is given in Section \ref{sec:iter}. Simulation results and discussion are provided in Section \ref{sec:Sim}. Finally, Section \ref{sec:concl}
presents conclusions and future directions.








\textit{Notation:}
The set of natural numbers is noted as $\mathbb N$. Let $\mathbb N_{[a,b]} =[a,b] \cap \mathbb N$.
For a signal $x$, its value at the discrete time $t=kT_s$ is denoted as $x(k)$, where $T_s$ is the sampling time. Let $\bar x(s|k)$ denote the predicted value of $x$ at $t=(k+s)T_s$ upon the feedback at time $t=kT_s$. The notation $Q\succeq0$ signifies that $Q$ is positive semi-definite. The set of $n\times n$ real symmetric positive definite matrices is noted by $\mathbb S_{++}^n.$

\vspace{-0.25cm}
\section{Preliminaries}
\label{sec:prelim}
\subsection{PMSM model with a finite control set}
Consider the model of a three-phase PMSM driven by a multi-level voltage-source inverter (VSI) \cite{reed2016simultaneous}:
\be 
\begin{aligned}
   & {\mathrm d}i_{dq}^r/{\mathrm d t} =
{A}
i_{dq}^r + 
{B}
u_{dq}^r + E, \ E = [0 \ \ \Lambda_{PM}]^\top,
\\
&A =\bbm -R/L_d \ & \omega_{\text{re}}L_q/L_d \\
- \omega_{\text{re}}L_d/L_q \ & -R/L_q 
\ebm , B =\bbm L_d^{-1}\ & 0 \\
0\ & L_q^{-1}\ebm, 
\end{aligned} \label{eq:dqrmodel_CT}
\ee
with $i_{dq}^{r} = [ i_d^{r} \ i_q^{r} ]^\top$
and $u_{dq}^{r} = [ u_d^{r} \ u_q^{r} ]^\top$
collecting the direct- and quadrature-axis currents (respectively voltages) expressed in the rotor reference frame. 
The stator resistance, permanent-magnet flux linkage, and electrical rotor speed are denoted by $R$, $\Lambda_{PM}$,  and $\omre$, respectively. 
The stator self-inductances associated with the $d$- and $q$-axes are $L_d$ and $L_q$, respectively.

\begin{figure}[thb]
    \centering
    \includegraphics[width=0.915\linewidth]{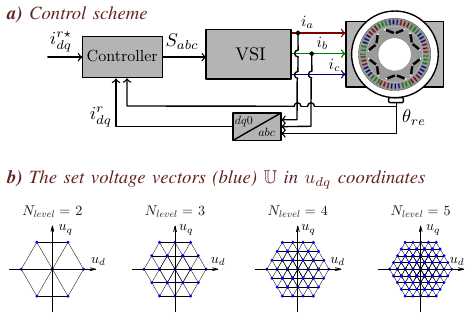}
    \caption{PMSM schematic and the finite control set in  \eqref{eq:fcs_alpha_beta}.}
    \label{fig:Scheme_directControl}
\end{figure}

Despite the deceptively simple linear structure, the dynamics \eqref{eq:dqrmodel_CT} pose a significant control challenge when powered by a VSI. Specifically, let us define the relation between the control input $u_{dq}^r$ and the VSI's discrete output as follows.
Let:
\be 
S_{abc} \in \mathbb N_{[0,N_{level}-1]}^3, \label{eq:swtState_abc}
\ee 
be the switch states of the inverter taking values from the finite set $ \mathbb N_{[0,N_{level}-1]}^3$ with $N_{level}$ representing the number of output voltage levels per phase. 
The inverter's output voltage in the stationary frame $u_{dq} = [u_d \ u_q]^\top$ is then computed as:
\be 
u_{dq} = {V_{bus}}({N_{level}-1})^{-1} T_{C}S_{abc}\triangleq  K(N_{level} )S_{abc}, \label{eq:clrkS}
\ee 
with $T_{C} =\frac 1 3 \bbms 2 &\ -1 &\ \  -1 \\
0 &\  \sqrt 3 &\  -\sqrt 3\ebms $ from the Clarke transform
and $V_{bus}$ being the bus voltage. The relation of {the inverter output voltage and its representation in the rotating frame $u_{dq}^r$ is:}
\be
u_{dq}^r = \bbm \phantom{-}\cos\ttre & \sin\ttre \\ -\sin \ttre & \cos \ttre \ebm u_{dq}   \triangleq {\mathsf R}(\ttre) u_{dq} \label{eq:clrk},
\ee 
where $\ttre$ is the electrical angle with $\dot\theta_{\mathrm{re}}=\omre$.
Consequently, the overall system is nonlinear and {switching} due to the transformation in \eqref{eq:clrk} and the finite control set in \eqref{eq:swtState_abc}, respectively.

To assess practical performance, we consider a sampled-data implementation.  In detail, given a sampling time $T_s$ (sec) with a zero-order hold actuation, the dynamics can be described as:
\begin{align}
   & i_{dq}^r(k+1) = A_di_{dq}^r(k) + B_d{\mathsf R}(\ttre(k))u_{dq}(k) + E_d ,  \nonumber\\
    &\ttre(k+1)  = \ttre(k) + \omega_{re}T_s, \label{eq:motor_discrete}
\end{align}
%
%
where assuming that Euler discretization is used: $A_d = T_sA + I$, $  B_d = T_s B, \ E_d = T_sE,$ and $u_{dq}$ belongs to the set:
\be 
\mathbb U =\{K(N_{level} ) S_{abc} \text{ as in \eqref{eq:clrkS}}: S_{abc} \in \mathbb N_{[0,N_{level}-1]}^3\}. \label{eq:fcs_alpha_beta}
\ee 
 In this setting, the inverter is limited to at most one switching event per sampling interval of length $T_s$. For later use, we define $
\mathbb U = \{u_1,u_2,...,u_{\mathrm p}\},
$ with $\mathrm p$ being the number of voltage vectors in $\mathbb U$ $\big($See Fig.~\ref{fig:Scheme_directControl}.b) for an illustration$\big)$.


\subsection{Optimal control formulation}
To optimize the current ripple during steady-state operation, we consider the minimization of the tracking error:
\be 
\xi(k) = i_{dq}^r(k) - i_{dq}^{r\star},
\ee 
with $i_{dq}^{r\star}$ being the desired current.
For this application (and for many switched systems in general), asymptotic convergence cannot be guaranteed, or $ \xi$ cannot be driven to 0 with limited switching {frequency}. 
Thus, we use the value function:
\be 
V^*(\xi(0),\ttre(0)) =\min _{{u}(k) \in \mathbb{U}}\textstyle\sum_{k=0}^{\infty} \gamma^k \|\xi(k)\|_2^2, \label{eq:org_infty_cost}
\ee 
which is computed using the error dynamics:
\be 
\begin{aligned}
    \pl \xi  & = A_d\ks \xi + B_d{\mathsf R}(\ks\ttre)\ks{u} + \bar x_e \\
    &\triangleq f(\ks\xi,\ks\ttre,\ks{u}).
\end{aligned}
\label{eq:error_dyna}
\ee 
with $\bar x_e = (A_d - I)i_{dq}^{r\star} + E_d $, and $\gamma \in(0,1)$ being the discount factor. For compactness, hereafter, the subscript $dq$ will be dropped from the control input.

With an infinite horizon, deriving $V^*$ in \eqref{eq:org_infty_cost} is often analytically intractable.
Thus, in the subsequent section, an approximate solution will be derived via the setting in \cite{wang2015approximate}.


\section{Value function approximation via iterated Bellman inequalities}
\label{sec:iter}

\subsection{Iterated Bellman inequalities}

{Under reasonable assumptions, the value function} $ V^*$ is a solution of the Bellman equation (see, e.g., \cite{bertsekas2012dynamic}) $\forall\xi,\ttre$:
\be 
    V^*(\xi,\ttre) = \min_{u\in\mathbb U}\ \{ \ell(\xi) + \gamma  V^*(f(\xi,\ttre,u),\ttre+\omega_{\mathrm{re}}T_s) \},
\label{eq:bellmanEq}
\ee 
with $ \ell(\xi) = \|\xi\|^2_2$. 
Existing approaches to solving \eqref{eq:bellmanEq} usually suffer from the curse of dimensionality.
To sidestep the problem, the condition \eqref{eq:bellmanEq} can be relaxed to finding $\hat V$ \cite{wang2015approximate}:
\be 
\hat V \leq \mathcal T \hat V, \label{eq:bellmanIneq_relaxed}
\ee 
where the operator $\mathcal T$ {is defined as}:
$$
(\mathcal T\hat V)(\xi) =\min_{u\in\mathbb U}\ \{ \ell(\xi) + \gamma \hat V(f(\xi,\ttre,u),\ttre+\omega_{\mathrm{re}}T_s) \},
$$
and $\hat V$ presents an underestimation of $V^*$ \cite{wang2015approximate}. For tractability, one can parameterize the $\hat V$ over a set of basis functions. However, such a parameterization will limit the solution to the subspace spanned by the basis functions. To reduce conservativeness, the Bellman inequality \eqref{eq:bellmanIneq_relaxed} is iterated as:
\be 
\hat V \leq \mathcal T^M\hat V \label{eq:bellmanIterationM}
\ee 
where $M>1$ is the number of iterations. This iteration is equivalent to searching for $\hat V_j$ satisfying \cite{wang2015approximate}:
\be 
\hat{V}_{j-1} \leq \mathcal T \hat V_j,\ j \in\mathbb N_{[1,M]}, \ \hat V_0 = \hat V_M ,\label{eq:Bellman_seq_compact} 
\ee 
and the approximate function will be $\hat V = \hat V_M.$
Next, we show that the conditions in \eqref{eq:Bellman_seq_compact} can be cast into a system of LMIs.
\subsection{Approximate dynamic programming via LMIs}
First, we restrict $\hat V_j$ to a function class parameterized by:
\be 
\hat V_j(\xi) = \xi^\top P_j\xi + 2q_j^\top \xi + r_j, \label{eq:basisQuad}
\ee 
with $P_j\in\mathbb S^2_{++}, q_j\in\R^2, r_j\in\R.$
Note that \eqref{eq:basisQuad} is an angle-independent parameterization. This is motivated by the fact that $\ttre$ is an uncontrolled variable; the effectiveness of this design choice will be shown in our results.
For compactness, let us denote the set of the parameters of $\hat V_j$ as $\Theta_j =\{P_j,q_j,r_j\}  $. Then, the condition  \eqref{eq:Bellman_seq_compact} 
requires $\forall  u\in\mathbb U$:
\be 
\hat V_{j-1}(\xi) \leq \ell(\xi) + \gamma \hat V_j(f(\xi,\ttre,u) ), \forall \xi, \ttre, j \in \mathbb N_{[1,M]} . \label{eq:Bellman_inter_constr}
\ee 

With such a reformulation, a sufficient condition for \eqref{eq:Bellman_inter_constr}  with a fixed $u\in\mathbb U$ is given in the following proposition.


\begin{prop}
\label{prop:fcs_theta_freeeee}
Given $u\in\mathbb U$, suppose  $\exists \tau\geq 0 $ such that
\be
G(\Theta_j,\Theta_{j-1}) \succeq \tau F(u),
\label{eq:compactLMI_perui}
\ee 
\end{prop}
with $G(\Theta_j,\Theta_{j-1})$ and $F(u)$ defined as
\be 
G(\Theta_j,\Theta_{j-1}) = \bbm Q_y & q_y\\ q_y^\top & r_y \ebm, F(u) = \bbms
0 & 0 & 0 \\
0 & -S & 0 \\
0 & 0 & \|u\|_2^2
\ebms ,
\ee 
{where} $C = [A_d\ I]$, $S = (B_d^\top B_d)^ {-1}$ and
\be 
\begin{aligned}
    Q_y & = \gamma C^{\top} P_j C + \bbms 
I-P_{j-1} & 0 \\
0 & 0
\ebms,\\
q_y&=\gamma (\bar{x}_e^{\top} P_j C + q_j^\top C)^\top - \bbm
q_{j-1} & 0 \\
\ebm^\top ,\\
r_y &= \gamma (\bar{x}_e^{\top} P_j \bar{x}_e+ 2 q_j^{\top} \bar{x}_e  + r_j) - r_{j-1}.
\end{aligned} \label{eq:someMatricesfortheLMI}
\ee 
Then, the { inequalities} \eqref{eq:Bellman_inter_constr} hold for the given $u\in\mathbb U,${$\ \forall \xi, \ttre $}.


\begin{proof}
First, for any $u\in \mathbb U$ and $\ttre\in[0,2\pi)$, define:
\be 
w = B_d {\mathsf R}(\ttre) u,\ z = A_d\xi + w +\bar x_e. \label{eq:weirdNotations}
\ee 
Then the condition \eqref{eq:Bellman_inter_constr} reads as:
\be 
\hat V_{j-1}(\xi) \leq \ell(\xi) + \gamma \hat V_j(z). \label{eq:compactBellmanIneq}
\ee 

Since the matrix $ {\mathsf R}(\ttre) $ is orthonormal in $\R^2$, one can state:
\be 
w^\top S w =w^\top (B_d^\top B_d)^ {-1} w \leq \|{\mathsf R}(\ttre) u\|_2^2 = \| u\|_2^2 .
\ee 
Therefore, for the fixed $u$, it suffices to ensure \eqref{eq:compactBellmanIneq} for all $\xi$
and all $w$ satisfying $w^\top S w \le \| u\|_2^2$, i.e.,
\be 
- w^\top S w  + \| u\|_2^2 \geq 0 \Rightarrow \mathcal Q(\xi,w) \geq 0 \label{eq:quadImplic}
\ee 
with the quadratic function $\mathcal Q(\xi,w)$ being:
\be \begin{aligned}
\mathcal Q(\xi,w) = \|\xi \|_2^2 + &\gamma (z^\top P_jz + 2q_j^\top z + r_j)\\
& - \xi^\top P_{j-1}\xi - 2q_{j-1}^\top \xi - r_{j-1}.
\end{aligned}
\ee 
Consequently, by re-arranging and with $z$ from \eqref{eq:weirdNotations}, the implication \eqref{eq:quadImplic} can be written as:
\be 
\psi^\top F(u) \psi \geq 0 \Rightarrow \psi^\top G(\Theta_j,\Theta_{j-1}) \psi \geq 0, \label{eq:implicationQuad}
\ee
with $\psi = [\xi^\top, w^\top, 1]^\top$.
Then, by the S-procedure \cite{boyd2004convex}, it suffices to find $\tau\geq 0 $ such that:
$
 G(\Theta_j,\Theta_{j-1}) \succeq \tau F(u).
$  
\end{proof}

With Proposition \ref{prop:fcs_theta_freeeee}, a solution of \eqref{eq:Bellman_seq_compact} can be computed via the following convex optimization problem:
\be 
\begin{aligned}
    \underset{\tau_i,\Theta_0,\Theta_j}{\max} {\mathfrak c}(\Theta_M)    
  \text{ s.t.}  \begin{cases}
 G(\Theta_j,\Theta_{j-1}) \succeq \tau_i F(u_i), \bbms P_0 & q_0 \\ q^\top_0 & r_0\ebms \succeq 0 ,&
  \\
      \Theta_0 = \Theta_M,\tau_i\geq 0, 
        i\in\mathbb N_{[1,\mathrm p]}, j\in\mathbb N_{[1,M]} & 
  \end{cases} 
\end{aligned}
\label{eq:sdp_full}
\ee 
with $\tau_i$ indexed only by $i$ and
$\mathfrak c(\Theta_M) $ being the user-defined objective function. 
For instance, one may wish to maximize the spatial average of $\hat V$ (or $\hat V_M$) over some operating region of the motor's currents, such as, over 
$ 
\mathcal B_\varepsilon =\{\xi:\|\xi\|_2\leq \varepsilon\},
$ {in which case} the objective function admits a closed-form expression:
\be \textstyle
\mathfrak c(\Theta_M)  =\frac{1}{\pi\varepsilon^2} \int_{\xi\in\mathcal B_\varepsilon} \hat V_M(\xi) \mathrm d \xi = \varepsilon^2{\text{Trace}(P_M)}/{4} + r_M. \label{eq:cost_volume}
\ee 

With such an approximation procedure, it should be noted that the setting involves a single hyperparameter $ M$. Ideally, $ M$ is chosen sufficiently large so that $\mathcal T^M\hat V$ approaches $V^*$ due to the contraction property $ \mathcal T$ (see \cite{bertsekas2012dynamic}). Thus, increasing $M$ reduces the conservativeness of the underapproximation, yet it linearly increases the computational cost of \eqref{eq:sdp_full}. Another factor related to the approximation quality is the S-procedure used to enforce \eqref{eq:implicationQuad}, as the LMI \eqref{eq:compactLMI_perui} is a sufficient condition for such an implication. This introduces additional conservativeness by excluding feasible solutions of $\Theta_j,\Theta_{j-1}$, a trade-off to make the problem convex, and hence tractable.

\subsection{Closed-loop analysis}\label{sec:closedloop}
We now present a controller and analyze the corresponding closed-loop properties. For brevity, hereinafter, we drop the subscript $M$ in the elements of $\Theta_M$ and refer to $\hat V = \hat V_M $ as
\be \hat V(\xi)= \xi^\top P \xi + 2q^\top \xi +r. \label{eq:reducedHatV}\ee 
Then, we use the horizon-one MPC defined as
\be
u=\mathbf k(\xi,\ttre) \triangleq {\arg\min}_{u\in\mathbb U}\; \ell(\xi) + \gamma\hat V(f(\xi,\ttre,u)). \label{eq:oneStepMPC}
\ee 
We refer to this controller as LMI-ADP and proceed with a verification for positive invariance.

\begin{lemma} \label{lem:PWQ}
    The controller \eqref{eq:oneStepMPC} can be written as:
\be 
\mathbf k(\xi,\ttre) = u_i, \text{ if } \xi \in \mathcal R_i(\ttre) ,\label{eq:Control_pwc}
\ee 
where $\mathcal R_i(\ttre) =\{\xi: a_{ij}^\top(\ttre) \xi \leq -b_{ij}(\ttre), \forall j\in\mathbb N_{[1,\mathrm p]}\}$ and with $h_i(\ttre) = B_d\mathsf R(\ttre)u_i+\bar x_e$, the halfspaces are:
\begin{align}
  &  a_{ij}(\ttre)= 2A_d^\top P \Tilde{h}_{ij}(\ttre), \Tilde{h}_{ij}(\ttre)=h_i(\ttre)-h_j(\ttre), \nonumber \\
  &  b_{ij}(\ttre) = \|h_{i}(\ttre)\|_P^2 - \|h_{j}(\ttre)\|_P^2 + {2}q^\top \Tilde{h}_{ij}(\ttre). \label{eq:hRep_ttre}
\end{align}
Moreover, with $\hat V^+(\xi,\ttre)\triangleq\hat V(f(\xi,\ttre,\mathbf{k}(\xi,\ttre)))$, one has:
\be \hat
 V^+(\xi,\ttre)= \hat V^+_i(\xi,\ttre), \text{ if } \xi \in \mathcal R_i(\ttre),\label{eq:Vplus}
\ee 
and $\hat 
V^+_i(\xi,\ttre) = \xi^\top \breve P \xi + {2}\breve q_i^\top(\ttre) \xi + \breve r_i(\ttre)
$ is defined by
$$ 
\begin{aligned}
&\breve P  = A_d^\top P A_d, \breve q_i(\ttre) =  A_d^\top P h_i(\ttre) + A_dq, \\
&  \breve r_i(\ttre)=  h_i(\ttre)^\top P h_i(\ttre) + {2}q^\top h_i(\ttre) + r .
\end{aligned}
$$ 
\end{lemma}
\begin{proof}
    Given $\ttre$, one can show $\forall j\in\mathbb N_{[1, \mathrm p]}:$ $$\hat V(f(\xi,\ttre,u_i)) \leq \hat V(f(\xi,\ttre,u_j))  \Leftrightarrow \xi\in\mathcal R_i(\ttre).$$
    Thus, \eqref{eq:Control_pwc} holds, which leads to \eqref{eq:Vplus} by substitution.
\end{proof}

For practical\footnote{We seek practical stability because for this class of discrete-time switched systems, one cannot achieve asymptotic stability \cite{deaecto2016practical}.} stability, we wish to certify that, 
for a given constant $\delta >0$ and $\Omega_\delta \triangleq \{\xi:\hat V(\xi)\leq \delta\}$, one has:
\be \hat
V^+(\xi,\ttre) \leq \delta, \ \forall\xi \in \Omega_\delta, \forall \ttre\in[0,2\pi).  \label{eq:invariant_cond}
\ee 
In other words, $\Omega_\delta$ is a positively invariant set, and the motor current will be bounded inside such a region. The following proposition formalizes such a verification setup.
\begin{prop} \label{prop:stab_veri}
   Given $\delta>0, $ let
   \be 
   \sigma_\delta(\ttre) = \max_{i\in\mathbb N_{[1,\mathrm p]}} \sigma_{\delta,i}(\ttre),\label{eq:sigma_def}
   \ee 
with $\sigma_{\delta,i}(\ttre) = \max \hat
V^+_i(\xi,\ttre) \text{ s.t } \xi\in\mathcal R_i(\ttre)\cap \Omega_\delta$ and $\max_{\ttre}\sigma_\delta(\ttre) \leq \delta$. Then, the set $\Omega_\delta$ is positively invariant.

\end{prop}
\begin{proof}
From \eqref{eq:Vplus} and \eqref{eq:sigma_def}, 
$\sigma_\delta(\ttre) = \displaystyle\max_{\xi \in\Omega_\delta}V^+(\xi,\ttre).$
Thus $\max_{\ttre}\sigma_\delta(\ttre) \leq \delta$ leads to \eqref{eq:invariant_cond}.
\end{proof}
Regarding verification complexity, the outer maximization is performed over the continuous one-dimensional variable $\ttre$, while each evaluation of $\sigma_{\delta,i}(\ttre)$ reduces to a smooth two-dimensional optimization problem over a convex set.

With such a set, the system's steady-state performance or the infinite-horizon average cost can be characterized as follows.

\begin{coro}\label{coro:AvgCost}
   Let $\Omega_\delta$ be an invariant set certified from Proposition \ref{prop:stab_veri}, and $  \nu \triangleq \max_{\ttre\in[0,2\pi]} \max_{\xi\in\Omega_\delta} \Gamma(\xi,\ttre)$ then
$$\textstyle
\bar J\triangleq\limsup_{N\to\infty} \frac{1}{N}\textstyle\sum_{k=0}^{N-1}\ell(\xi(k))\leq\nu, 
$$
{with }  $\Gamma(\xi,\ttre)\triangleq  \hat V^+(\xi,\ttre) - \hat V(\xi) + \ell(\xi(k))$.\end{coro}

\begin{proof}
 By construction, one has 
\be 
\hat V^+(\xi,\ttre) - \hat V(\xi) \leq -\ell(\xi(k)) + \nu,\ \forall \ttre,\forall \xi\in\Omega_\delta.
\label{eq:ISS_cond}
\ee 
Suppose $\xi(0)\in\Omega_\delta$, from \eqref{eq:ISS_cond}, by telescoping 
\begin{align}
&\textstyle\sum_{k=0}^{N-1} \left[  \hat V(\xi({k+1})) - \hat V(\xi(k)) \right]   \leq  \sum_{k=0}^{N-1} \left[ - \ell(\xi(k)) + \nu \right],\nonumber  \\
   &\textstyle\Leftrightarrow \frac{1 }{N }\sum_{k=0}^{N-1}\ell(\xi(k))  \leq \nu  - \frac{1 }{N }\Big[\hat V(\xi(N)) - \hat V(\xi(0)) \Big] \nonumber \\
    & \textstyle\Rightarrow\limsup _{N \to \infty} \frac{1}{N} \sum_{k=0}^{N-1}\ell(\xi(k))  \leq \nu .\label{eq:Avg_cost_1}
\end{align}
%
\end{proof}
The metric $\bar J$ is relevant because it represents an upper bound on the current ripple in steady state operation.
\section{Simulation results}
\label{sec:Sim}
{In this section, we report and discuss the simulation
results.}



    
    

\subsection{Numerical setup}\label{subsec:numSet}
For simulations, the model's parameters {of a representative PMSM intended for EV applications} are extracted from the finite-element analysis model using \texttt{Ansys} software and are given as: $L_d = 0.46, L_q = 1.40 $ (mH), $\Lambda_{PM} = 0.1539$ (Wb), and $ R = 0.032$ ($\Omega$). The electrical speed is computed as $\omre = {\pi\omega_cN_{pole}}/{60}$, with $N_{pole}=6$ being the number of magnetic poles, and the commanded motor speed is given as $\omega_c = 15000$ (RPM). The bus voltage $ V_{bus} = 800V$ and the desired current vector is $i_{dq}^{r\star} = [-194.59\ \ 52.53]^\top $ (A).
For analysis {of the inverter design choices}, the synthesis of $\hat V$ will be carried out for $N_{level} \in\mathbb N_{[2,5]}$.  The semi-definite program \eqref{eq:sdp_full} is solved with MOSEK via YALMIP interface. The number of iterations is set as $M=100, \gamma = 0.975$ with the cost function in \eqref{eq:cost_volume}.



For comparison,
with the indirect architecture (controller-modulator), we employ a proportional-integral (PI) controller to generate the reference voltage $u_{dq}^r$ with a bandwidth of $180\pi$~(rad/s), which is then synthesized via space vector modulation (SVM) for the two-level inverter~\cite{handley1990space}. This setup primarily captures the effect of the fixed switching frequency in SVM.
For the direct control scheme, we consider the FCS MPC (FCS-MPC) law $u(k)=\bar u^\star(0|k)$. The sequence $\bar u^\star$ is the minimizer in the following optimization problem: 
\be 
\begin{aligned}
    \bar u^\star &= \underset{\bar u\in\mathbb U}{\arg \min } \  \textstyle\sum_{s=0}^{N_p}\ell(\bar \xi(s|k))  \\
    \text{s.t } & \bar \xi(s+1|k) = f(\bar \xi(s|k),\bar{\theta}_{\mathrm{re}} (s|k), \bar u(s|k) ), \\
    &
    \bar{\theta}_{\mathrm{re}}(s+1|k) =  \bar{\theta}_{\mathrm{re}}(s|k) + \omre Ts ,\\
    & \bar \xi(0|k) = \ks{ i_{dq}^r} - i_{dq}^{r\star}, \bar{\theta}_{\mathrm{re}}(0|k) = \ks\ttre.
\end{aligned}
\label{eq:FCSMPC}
\ee 

To {illustrate} the trade-off between switching effort and tracking performance, we use the definition of total harmonic distortion (THD) {which is a standard metric used to assess current ripple} \cite{karamanakos2019guidelines}. To quantify the {switching} effort, the number of switching instants (referred to as $S_c$ or Switch Count) will be counted, as it is directly related to the energy switching loss of the power electronic components \cite{geyer2016model}. Particularly, for the case of a two-level inverter, we convert the switching effort into the notion of ``switching frequency" $f_{sw}$ within the SVM framework, where, in one switching period, each of the three phases experiences two switching events. Hence, $f_{sw} = S_c/(6t_{sim})$(Hz) where $t_{sim}$ (sec) is the simulation time. The sampling frequency is denoted as $f_s=1/T_s$ and will be a variable in the analysis.
Finally, in particular for the two-level inverter, we employ the value iteration method proposed in \cite{bertsekas2012dynamic} to obtain the ``ground-truth\footnote{{Here, “ground truth” denotes the value function obtained via value iteration on a gridded state space, serving as a numerical reference for comparison.}}" for $V^*$. The domain of interest $\mathcal E = \{\xi:|\xi |\leq 300 \text{(A)}\}$ is uniformly gridded with the resolution of $5$ (A) while $\ttre \in \left\{(k-1)\pi/15: k\in\mathbb N_{[1,30]}\right\}$.
All computations are performed in MATLAB 2025b on an Intel Core i5-10300H CPU @ 2.50GHz with 16GB RAM.

\subsection{Results and discussions} \label{sec:SimRes}
\begin{figure}[htbp]
    \centering
    \includegraphics[width=0.995\linewidth]{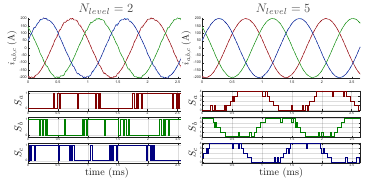}
    \caption{Horizon-one MPC with the approximate value function for different values of $N_{level}$ with $f_s = 50$kHz.}
    \label{fig:HorizonOne}
\end{figure}

The simulation results for inverters of different levels are provided in Fig. \ref{fig:HorizonOne} and Table \ref{tab:Simulation} where the sampling frequency of 50kHz has been assumed. 
As observed, the THD decreases as the number of inverter voltage levels increases. This behavior is consistent with the expected trade-off between hardware complexity and achievable performance. From another perspective, this effect can also be interpreted as a quantization phenomenon, where increasing the number of voltage levels effectively refines the discretization of the control set.


In terms of stability certification, we report the case $N_{level} = 2$ in Fig. \ref{fig:Stability_verifi} corresponding to the smallest value of $\delta=1122.75$, for which Proposition \ref{prop:stab_veri} remains satisfied. {The function $\sigma_\delta(\ttre)$ in \eqref{eq:sigma_def} is computed via \texttt{fmincon} method in MATLAB with multiple initial guesses. We also note that, as the problem therein is only two-dimensional, a global solution can also be characterized by enumerating all candidate solutions from the Karush–Kuhn–Tucker conditions.} We then examine forward invariance by evaluating the closed-loop vector field along the boundary of $\Omega_\delta$ for $\ttre\in[0,2\pi)$. It can be observed that all successor states lie inside the invariant set and the average cost is bounded by $\nu$ in \eqref{eq:Avg_cost_1}.
\begin{table}[hbt]
    \centering
    \caption{Results for different $N_{level}$ with \eqref{eq:oneStepMPC} \& $f_s=50kHz$.}\begin{tabular}{|c|c|c|c|c|} \hline
       $N_{level}$  & 2 & 3 & 4 & 5\\  \hline 
        THD & 2.93\% & 1.41\% & 0.95\%  & 0.76\%\\ \hline
        $S_c$ & 85 & 87 & 105 & 104\\\hline
        CT-offline (sec) & 4.4  & 9.5 & 17.2 & 29.9\\\hline
        mean CT-Online (ms) & 0.36 & 0.38 & 0.92 & 1.62\\\hline
    \end{tabular}
    
    \label{tab:Simulation}
\end{table}

Regarding the offline computational effort required to approximate the value function, we report that the value iteration approach requires approximately 55 minutes to converge for the simplest VSI configuration ($N_{level} = 2$). In contrast, the proposed LMI-based formulation completes in about 4 seconds for the two-level {inverter} case and around 30 seconds for the five-level inverter {case}. These results further highlight the computational advantage of the proposed convex formulation.
{This computational advantage becomes particularly relevant in {situations} where the value function must be recomputed repeatedly, such as in higher-level design or co-design procedures involving system parameters.}

\begin{figure}
    \centering
    \includegraphics[width = 0.995\linewidth]{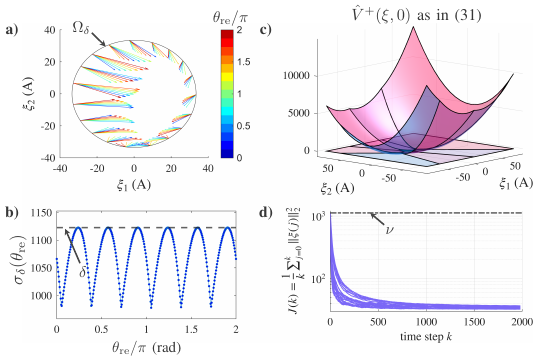}
    \caption{a) The vector field with different values of $\ttre$ on the boundary of $\Omega_\delta$; b) The function $\sigma_\delta(\ttre)$; c) The piecewise quadratic function $\hat V^+(\xi, \ttre)$ defined over $\mathcal R_i(\ttre)$ as in \eqref{eq:Vplus} with $\ttre=0$; d) Upper bound on the average current ripple.}
    \label{fig:Stability_verifi}
\end{figure}

\begin{figure}[tbp]
    \centering
    \includegraphics[width=0.95\linewidth]{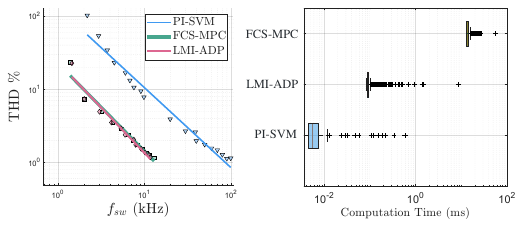}
    \caption{Top: Comparison of performance trade-off between the proposed setting, FCS-MPC ($N_{p} = 3$) and PI-SVM.}
    \label{fig:Frontiers_compare}
\end{figure}

Regarding the online performance of the controller \eqref{eq:oneStepMPC}, its computational burden reduces to evaluating a two-dimensional quadratic function over the finite set $\mathbb U$. The reported computation times are obtained from a MATLAB implementation, yielding  90$\mu s$ on average\footnote{{We note that the value 0.36 ms in Table \ref{tab:Simulation} corresponds to the benchmark case reported in therein, while the 
value in Fig. \ref{fig:Frontiers_compare} corresponds to the average time spent per control update over all tested sampling frequencies $f_s$.}} for LMI-ADP compared to $5.5\mu s$ for PI-SVM and $13400\mu s$ of FCS-MPC (see Fig. \ref{fig:Frontiers_compare}, top-right).
We note that, in embedded implementations, such a search over the finite set can typically be accelerated to the microsecond level \cite{stellato2016high}. 
Next, tracking efficiency results are reported in Fig. \ref{fig:Frontiers_compare}, where the proposed LMI-ADP scheme is compared with PI-SVM and FCS-MPC. First, the proposed approach yields a significantly more favorable trade-off between switching effort and THD than PI-SVM. This behavior is expected since SVM enforces a fixed switching pattern with six commutations per period (in the two-level case), whereas our variable-switching-frequency policy can maintain the same switching state when appropriate, thereby avoiding unnecessary switching events.
Second, this advantage is also reflected in the trade-off curve obtained with FCS-MPC. In the present example, the performance of the proposed LMI-ADP scheme is comparable to that of FCS-MPC while requiring significantly lower online computational effort. This again highlights the benefit of precomputing the value function and using it as a tail cost in the online OCP.

\begin{figure}[htbp]
    \centering
     \includegraphics[width = 0.95\linewidth]{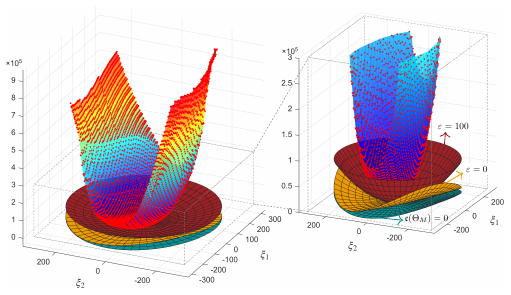}
    \caption{The approximation function $\hat V$ from different objective functions in \eqref{eq:sdp_full} and the interpolated function from value iteration (red points) with $\ttre=0$ {(The right figure shows a close-up of the boxed region)}.}
    \label{fig:ValueIter_compare}
\end{figure}

This observation naturally raises the question of the quality of the under-approximation of the value function. To visualize this, Fig. \ref{fig:ValueIter_compare} shows the function obtained from value iteration over a dense grid, plotted as red points. In the figure, we vary the choice of the cost function $\mathfrak c(\Theta_M)$ in \eqref{eq:sdp_full} to illustrate its effect on the approximation.
More specifically, when the cost is chosen as 
$\mathfrak c(\Theta_M)=0$ and only the LMI conditions are enforced as hard constraints, the resulting approximation is the most conservative, as no mechanism encourages the under-approximation to increase. Next, setting 
$\varepsilon=0$ promotes only the free scalar variable 
$r_M$, leading to a less conservative result. Finally, when a region of interest is introduced with 
$\varepsilon=100$ (A), the approximation becomes the closest to the value function obtained via value iteration.
However, as visible from the plot, the approximation remains far from tight. This limitation stems from the limited expressiveness of the quadratic basis functions in \eqref{eq:basisQuad}, which cannot capture more complex structures. This observation motivates future investigations into more flexible approximating structures, such as polynomials \cite{heydari2016optimal} or piecewise quadratic functions \cite{he2024approximate}.
{Note that, nevertheless, in terms of performance, our approach shows similar results to FCS-MPC (Fig. \ref{fig:Frontiers_compare}-left), which is expected to approximate DP for longer horizons yet at a lower computation cost (Fig. \ref{fig:Frontiers_compare}-right).}


{Finally, while the simulations demonstrate a desirable performance improvement with the proposed approach, there are several limitations of the current work. In particular, the method is developed based on a nominal PMSM model. In practice, parameter variations and uncertainties may affect performance. Robustness to such unmodeled factors, as well as experimental validation, is not covered in this work and remains a direction for future investigation.}





\section{Conclusion}
\label{sec:concl}
This paper presented an LMI-based approach for approximating the infinite-horizon value function in finite-control-set problems for three-phase PMSMs with rotational dynamics. The proposed formulation enables the value function to be computed orders of magnitude faster than conventional value iteration. 
Simulation results on multilevel inverters show that the resulting closed-loop performance achieves a favorable trade-off between THD and switching effort, comparable to FCS-MPC but with a significantly lower online computational burden. We have highlighted the benefit of precomputing the value function and using it as a tail cost in a horizon-one MPC setting. In addition to the experimental validation, future work will investigate richer classes of as well as their integration into the bilevel problem of system and control co-design. 

\end{document}